\documentclass[11pt]{article}
\usepackage{amssymb,amsmath,amsthm,amscd,amsfonts }
\usepackage{enumerate}

  {\noindent{\bf\ref{#2}#1.}}{\vspace{\baselineskip}}

\makeatletter
\def\ulabel#1#2{\@bsphack\if@filesw {\let\thepage\relax \def\protect{\noexpand\noexpand\noexpand}%
\xdef\@gtempa{\write\@auxout{\string
\newlabel{#1}{{#2 \@currentlabel}{\thepage}}}}}\@gtempa
\if@nobreak \ifvmode\nobreak\fi\fi\fi\@esphack} \makeatother

\newcommand {\N}{\mathbb{N}} 
\newcommand {\R}{\mathbb{R}} 

\theoremstyle{plain}
\newtheorem{thm}{Theorem}[section]
\newtheorem{thmm}{Theorem}

\newtheorem{cor}[thm]{Corollary}

\newtheorem{lem}[thm]{Lemma}
\newtheorem{prop}[thm]{Proposition}

\newtheorem*{clm*}{Claim}
\newtheorem*{rmk*}{Remark}
\newtheorem*{thm*}{Theorem}
\newtheorem{quest}{Question}

\theoremstyle{definition}
\newtheorem{defn}[thm]{Definition}
\newtheorem{rmk}[thm]{Remark}

\newcommand{\mb}[1]{\mathbb{#1}}
\newcommand{\mesh}{\operatorname{mesh}}

\newcommand{\pres}[2]{\langle\hspace{.5 mm} #1\hspace{.5 mm} | \hspace{.5 mm} #2\hspace{.5 mm} \rangle}
\newcommand{\con}{\operatorname{Con}^\omega\bigl(X,e,d\bigr)}
\newcommand{\conn}[1]{\operatorname{Con}^\omega\bigl(X,e,(#1)\bigr)}
\newcommand{\gcon}{\operatorname{Con}^\omega\bigl(G,d\bigr)}

\renewcommand{\d}{\operatorname{dist}}
\newcommand{\diam}[1]{\operatorname{diam}(#1)}

\newcommand{\area}[1]{\operatorname{Area}(#1)}
\renewcommand{\phi}{\varphi}
\newcommand{\im}{\operatorname{im}}

\newcommand{\lab}[1]{\operatorname{\textbf{Lab}}{(#1)}}

\begin{document}
\title{Local topological properties of  asymptotic cones of groups}
\author{Greg Conner and Curt Kent}
\maketitle

\begin{abstract}
We define a local analogue to Gromov's loop division property which is use to give a sufficient condition for an asymptotic cone of a complete geodesic metric space to have uncountable fundamental group.  As well, this property is used to understand the local topological structure of asymptotic cones of many groups currently in the literature.
\end{abstract}

\tableofcontents

\section{Introduction}

Gromov \cite[Section 5.F]{gr2} was first to notice a connection between the homotopic
properties of asymptotic cones of a finitely generated group and
algorithmic properties of the group: if all asymptotic cones of a finitely generated group are simply connected, then the group is finitely presented, its Dehn function is
bounded by a polynomial (hence its word problem is in NP) and its
isodiametric function is linear. A version of that result for higher
homotopy groups was proved by Riley \cite{Ri}. The converse statement does
not hold: there are finitely presented groups with non-simply connected
asymptotic cones and polynomial Dehn functions \cite{Bridson}, \cite{SBR},
and even with polynomial Dehn functions and linear isodiametric functions
\cite{OS1}. A partial converse statement was proved by Papasoglu
\cite{pap}: a group with quadratic Dehn function has all asymptotic cones
simply connected (for groups with subquadratic Dehn functions, i.e.
hyperbolic groups, the statement was previously proved by Gromov
\cite{gr1}: all asymptotic cones in that case are $\R$-trees). An example
of Thomas and Velickovic \cite{TV} shows that for a finitely generated
group one asymptotic cone can be a tree (and hence simply connected) while
another asymptotic cone may have non-trivial $\pi_1$.  In Section \ref{section examples}, we show how to modify Thomas and Velickovic's example to obtain a finitely generated group with one asymptotic cone which is an $\mathbb R$-tree and one asymptotic cone which is not locally simply connected.  Thus finitely generated groups can have asymptotic cones which are not locally bi-Lipschitz.

If a group is finitely presented and one asymptotic cone is an $\R$-tree, then the group is
hyperbolic, so all asymptotic cones are simply connected (it essentially
follows from Gromov's version of the Cartan-Hadamard theorem for hyperbolic groups, see the appendix
of \cite{OOS}). Nevertheless in \cite{OS2}, a finitely presented group (a multiple HNN extension of a free group) having both simply connected and non-simply connected asymptotic cones was constructed.

In \cite{gr2}, Gromov defined a loop division property and outlined a proof that a metric space has the loop division property if and only if all of its asymptotic cones are simply connected.  Papasoglu presented a proof of the \textit{only if} direction in \cite{pap}.  Dru\c tu gave a proof of the \textit{if} direction in \cite{dru}. A version of the loop division property which guarantees that a particular asymptotic cone is simply connected was presented and used by Olshanskii and Sapir in \cite{OS2}.  Here we will define an analogue Gromov's loop division property which we will use to understand the local topological structure of asymptotic cones.

In Section \ref{section gldp}, we recall some of the definitions and consequences of Gromov's loop division property as studied by Papasoglu and define a local version which we call \emph{$\epsilon$-coarsely loop divisible}.  The coarsely loop divisible property depends on a scaling sequence and an ultrafilter.  We prove that a space is $\epsilon$-coarsely loop divisible with respect to a pair $(\omega, d\bigr)$ if and only if all sufficiently short loops in $\con$ can be partitioned into strictly shorter loops.  We say that a space is uniformly $\epsilon$-coarsely divisible if the number of piece required to partition small loops in $\con$ is uniformly bounded independent of the chosen loop.

\begin{thmm}[\ref{eldpsimply}, \ref{slsc}, \ref{simply}]

Let $G$ be a finitely generated group and fix a pair $\bigl(\omega, d\bigr)$.

\begin{enumerate}[1)]
    \item If $G$ is uniformly $\epsilon$-coarsely loop divisible, then $\gcon$ is uniformly locally simply connected and $G$ has an asymptotic cone which is simply connected.
    \item If $\gcon$ is semi-locally simply connected, then $G$ is $\epsilon$-coarsely loop divisible.
\end{enumerate}

\end{thmm}

Papasoglu (see \ref{pap2}) showed that if one requires $G$ to be \emph{uniformly $\epsilon$-coarsely loop divisible with respect to $\bigr(\omega, d\bigl)$ for every $\epsilon>0$}, then one obtains that $\gcon$ is actually simply connected.  However; it is not clear if uniformly coarsely divisible is actually a necessary condition.  Hence, the following questions are open.

Let $G$ be a finitely generated group.
\begin{quest}\label{quest1a} If $\gcon$ is locally simply connected, is $G$ uniformly $\epsilon$-coarsely loop divisible? \end{quest}

\begin{quest}\label{quest1} If $\gcon$ is simply connected, is $G$ uniformly $\epsilon$-coarsely loop divisible for every $\epsilon$? \end{quest}

\ref{hawaiianearing} gives examples of metric spaces which are not asymptotic cones where the answer to both of these question is no.  There are no known examples of finitely generated groups which are coarsely loop divisible but not uniformly coarsely loop divisible which leaves the following question open.

\begin{quest}\label{quest1b}
Are uniformly coarsely loop divisible and coarsely loop divisible equivalent conditions for finitely generated groups?
\end{quest}

A positive answer to Question \ref{quest1b} would imply a positive answer to Question \ref{quest1a} and show that for finitely generated groups locally simply connected and semi-locally simply connected are equivalent properties.

Coarse loop divisibly also allows us to understand some general algebraic properties of the fundamental group of an asymptotic cone.

\begin{thmm}[\ref{uncountable}, \ref{notfree}, \ref{notsimple}]\label{introthm}\ulabel{introthm}{Theorem}
If a finitely generated $G$ is not $\epsilon$-coarsely divisible with respect to $\bigr(\omega,d\bigl)$ for every $\epsilon>0$, then the fundamental group of $\gcon$ is uncountable, not free, and not simple.
\end{thmm}

These theorems hold for all complete homogenous geodesic metric spaces. In Section \ref{subsection absolutely non-divisible}, we give a necessary condition for every asymptotic cone of a complete homogenous geodesic metric space to satisfy the hypothesis of \ref{introthm}.  It turns out that many important groups such as $SL_3 (\mathbb Z)$ and other groups that have previously appeared in the literature related to asymptotic cones satisfy this condition, see Section \ref{section examples}.

\subsection{Definitions}

Let $G=\langle S\rangle$ be a group and $u,v$ be two words in the alphabet $S$. We write $u\equiv v$ when $u$ and $v$ coincide letter by letter and $u=_G v$ if $u$ and $v$ are equal in $G$.  We will denote the Cayley graph of $G$ with respect to the generating set $S$ by $\Gamma(G,S)$.  We will use $\operatorname{\textbf{Lab}}$ to represent the function from the set of edge paths in a labeled oriented CW complex to the set of words in the alphabet obtained by reading the label of a path.

\textbf{Isoperimetric functions:}  Suppose that $\pres{S}{R}$ is a finite presentation for a group $G$.  Let $\area\Delta$ denote the number of $R$-cells in a van Kampen diagram $\Delta$.  If $w$ is a word in $S\cup S^{-1}$, then $\area w = \min \{\area \Delta \ | \ \lab{\partial\Delta} \equiv w \}$.  If $\gamma$ is a loop in $\Gamma(G,S)$, then $\area\gamma = \area{\lab{\gamma}}$.

An isoperimetric function for the presentation $\pres{S}{R}$ of $G$ is a non-decreasing function $\delta: \N \to [0,\infty)$ such that $\delta(|\partial \Delta|)\geq \area{\lab{\partial \Delta}}$ for all van Kampen diagrams $\Delta$ over $\pres{S}{R}$.  A minimal isoperimetric function of a group is called a \emph{Dehn} function for $G$.

Two non-decreasing functions $f, g:\N \to [0,\infty)$ are \emph{equivalent}, if there exists constants $B,C>0$ such that $f(n) \leq  Bg(Bn +B )+ Bn+B \leq Cf(Cn+C) +Cn +C$.  Up to this equivalence, the Dehn function of a finitely presented group is independent of the finite presentation.

\begin{defn}[Asymptotic cones]
Let $\omega$ be an ultrafilter on $\N$ and  $c_n$ be a sequence of positive real numbers.  The sequence $c_n$ is \emph{bounded $\omega$-almost surely} or \emph{$\omega$-bounded}, if there exists a number $M$ such that $\omega\bigl(\{n \ | \ c_n< M\}\bigr) = 1$.  If $c_n$ is $\omega$-bounded, then there exists a unique number, which we will denote by $\lim^\omega c_n$, such that  $\omega\bigl(\{ n \  | \ |c_n- lim^\omega c_n|< \epsilon\}\bigr) = 1$ for every $\epsilon>0$.

If $c_n$ is not $\omega$-bounded, then $\omega\bigl(\{n \ | \ c_n> M\}\bigr) = 1$ for every $M$.  We will say that \emph{$c_n$ diverges $\omega$-almost surely} or is \emph{$\omega$-divergent} and let $\lim^\omega c_n = \infty$.

Let $(X,\d)$ be a metric space.  Let $\omega$ be an ultrafilter on $\N$.  Consider an $\omega$-divergent sequence of numbers $d = (d_n)$ called a \emph{scaling sequence} and a sequence of points $e=(e_n)$ in $X$ called an \emph{observation sequence}.

Given two sequences $x = (x_n), y = (y_n)$ in  $X$, set $\d(x,y) = \lim{}^\omega\dfrac{\d(x_n,y_n)}{d_n}$.  We can then define an equivalence relation $\sim$ on the set of sequence in $X$  by  $x\sim y$ if and only if $\d(x,y)= 0$.

The asymptotic cone of $X$ with respect to $e$, $d$, and $\omega$ is

$$ \operatorname{Con}^\omega\bigl(X,e,d\bigr) = \{ x = (x_n) \ | \ \d(x,e)<\infty \}/\sim.$$

$\operatorname{Con}^\omega\bigl(X,e,d\bigr)$ is a complete metric space. If $X$ is geodesic, then $\operatorname{Con}^\omega\bigl(X,e,d\bigr)$ is also geodesic.

If $X_n$ is a sequence of subspaces of $X$, we will use $\lim^\omega X_n$ to denote the subspace of $\operatorname{Con}^\omega\bigl(X,e,d\bigr)$ consisting of sequences with representatives in $\prod X_n$.

\end{defn}

The following lemma is obvious.

\begin{lem}
Let $\omega$ be an ultrafilter on $\N$ and $d= (d_n)$ be an $\omega$-divergent
sequence of numbers. Suppose that $\{\gamma_n\}$ is a sequence of paths parameterized by arc length in a geodesic
metric space $(X, \d)$ such that $|\gamma_n| = O(d_n)$. Then $\gamma(t) =
\bigl(\gamma_n(t)\bigr)$ is a continuous map into $\con$.
\end{lem}

The following converse holds and is proved in \cite{Kent}.

\begin{lem}\label{limitloop}\ulabel{limitloop}{Lemma}
Suppose that $\gamma$ is a path in $\con$ where $X$ is a geodesic metric space.  Then there exist paths $\gamma_n$ in $X$ such that $\gamma(t) = \bigl(\gamma_n(t)\bigr)$.
\end{lem}

One should take care to understand that \ref{limitloop} does not imply that all geodesics are limits of geodesics which is not true.  Throughout this paper, we will assume that metric balls are open.  When $\tau$ is a path in a metric space, we will use $|\tau|$ to denote its arc length. Then $|\cdot|$ maps the set of paths into the extended real line and is finite for rectifiable paths and $+\infty$ for
non-rectifiable paths. We will assume that rectifiable paths are parameterized proportional to arc length.

The following definitions of locally connectivity properties are standard, see \cite[Chapter 1]{Hatcher}.
\begin{defn}

A space $X$ is called \emph{locally simply connected} if for every pair $(U,x)$ where $U$ is a neighborhood of $x\in X$, there exists  $V$, a neighborhood of $x$ contained in $U$, such that the inclusion induced homomorphism from $\pi_1(V,x)$ to $\pi_1(U,x)$ is trivial; i.e. every loop in $V$ bounds a disc in $U$.  A metric space $X$ is \emph{uniformly simply connected} if for every $\epsilon>0$ there exists a $\delta>0$ such that every loop with diameter at most $\delta$ bounds a disc with diameter at most $\epsilon$.

A space $X$ is called \emph{semilocally simply connected} if every point $x\in X$ has a neighborhood $U$ such that the inclusion induced
homomorphism from $\pi_1(U,x)$ to $\pi_1(X,x)$ is trivial, i.e. every loop in $U$ bounds a disc in the whole space.

\end{defn}

\begin{rmk*}
A space that is locally simply connected is semilocally simply connected.
The converse is false, since the cone on any space that is not locally
simply connected is semilocally simply connected but still not locally simply
connected.  See \cite[Section 1.3]{Hatcher}.\end{rmk*}

The following definition of a \emph{partition} is due to Papasoglu \cite{pap}.

\textbf{Partitions of the unit disc in the plane:} Let $\mathbb D$ be the unit disk in $\mathbb R^2$. A \emph{partition $P$ of $\mathbb D$} is a finite collection of closed discs $D_1, \cdots, D_k$ in the plane with pairwise disjoint interiors such that $\mathbb D = \cup_i D_i$, $\partial \mathbb D = \partial (D_1\cup \cdots\cup D_k)$, and $D_i\cap D_j = \partial D_i\cap\partial D_j$ when $i\neq j$. A point $p$ on $\partial D_1\cup \cdots\cup\partial D_k$  is called a \emph{vertex of the partition} if for every open set $U$ containing $p$, $U\cap (\partial D_1\cup \cdots\cup\partial D_k)$ is not homeomorphic to an
interval. An \emph{edge of a partition} is a pair of vertices which are joined by a path in $\partial D_1\cup \cdots\cup\partial D_k$ that intersects the set of vertices only at its endpoints.  We will say that such vertices are \emph{adjacent}.  A \emph{piece of a partition} is a maximal set of vertices of the partition contained in a single disc of the partition.  A partition is then a cellular decomposition of the unit disc where each vertex has degree at least 3; so we will use the standard notation, $P^{(i)}$, to denote the $i$-th skeleton of a partition for $i= 0,1,2$.

\textbf{Geodesic $n$-gons in a metric space X:} An \emph{$n$-gon} in $X$ is a
map from the set of vertices of the standard regular $n$-gon in the plane
into $X$, i.e. an ordered set of $n$ points in $X$.  If $X$ is a geodesic
metric space, we can extend an $n$-gon to edges by mapping the edge
between adjacent vertices of the standard regular $n$-gon in the plane to a geodesics segment joining the corresponding vertices of the $n$-gon in $X$.  We will say that such an extension is a \emph{geodesic $n$-gon} in $X$.

\textbf{Partitions of loops in a geodesic metric space X:} Let $\gamma:\partial\mathbb D\to X$ be a continuous map.
A \emph{partition of $\gamma$} is a map $\Pi$ from the set of vertices of a partition $P$ to $X$ such that $\Pi\bigl|_{\partial \mathbb D\cap P^{(0)}}=\gamma\bigl|_{\partial \mathbb D\cap P^{(0)}}$. The \emph{vertices/edges/pieces} of $\Pi$ are the images of
vertices/edges/pieces of $P$.  We will write $\Pi(\partial D_i)$ for the
pieces of $\Pi$, where $D_i$ are the $2$-cells of $P$.

\begin{rmk}\label{extend}\ulabel{extend}{Remark}
Suppose that $\Pi:P^{(0)}\to X$ is a partition of a loop $\gamma$ in a geodesic metric space.  We can extend $\Pi$ to $P^{(1)}$ by mapping each edge contained in $\partial \mathbb D$ to the corresponding subpath of $\gamma$ and every edge not contained in $\partial \mathbb D$ to a geodesic segment joining its end points.  The \emph{length of a piece} is the arc length of the loop $\Pi(\partial D_i)$.  We will write $|\Pi(\partial D_i)|$ for the length of the piece $\Pi(\partial D_i)$. We define the \emph{mesh of $\Pi$} by $$\mesh (\Pi) = \max\limits_{1\leq i\leq k} \{|\Pi(\partial D_i)|\}.$$
\end{rmk}

When $X$ is a Cayley graph of a group, we will also assume that the partition takes vertices of $P$ to vertices in the Cayley graph.  A partition $\Pi$ is called a \emph{$\delta$-partition}, if $\mesh \Pi<\delta$.  A loop of length $k$ in a geodesic metric space is \emph{partitionable} if it has a $\frac{k}{2}$-partition.

Let $P(\gamma,\delta)$ be the minimal number of pieces in a $\delta$-partition of $\gamma$ if a $\delta$-partition exist and $+\infty$ otherwise.

\section{Coarse Loop Division Property}\label{section gldp}

\begin{defn}\label{b}
Let $X$ be a geodesic metric space.

Define $\vartheta^i:\N \to \N\cup\{\infty\}$ by $\vartheta^i(n)= \sup\bigl\{P(\alpha, \frac{|\alpha|}{2^i}) \ | \ \alpha \text{ is a loop in } $X$ \text{ such that } n-1<|\alpha|\leq n\bigr\}$. We will call $\vartheta=\vartheta^1$ the \emph{divisibility} function of $X$.

Suppose $\omega$ is an ultrafilter on $\mb N$, $(d_n)$ an $\omega$-divergent sequence of positive real numbers, and $\epsilon$ a positive real number. We will say that $X$ is \emph{$ \epsilon$-coarsely loop divisible}; if for every $\delta\in(0,\epsilon)$ there exists an $A\subset \N$ with $\omega(A)=1$ such that the divisibility function $\vartheta$ restricted to $ \bigcup\limits_{n\in A}[\delta d_n, \epsilon d_n]$ is bounded by a constant $K=K(\delta,\epsilon)$.

We will say that $X$ is \emph{uniformly $\epsilon$-coarsely loop divisible}; if the constant $K=K(\delta, \epsilon)$ can be chosen independent of $\delta$.

We will say that a group $G$ is (uniformly) $\epsilon$-coarsely loop divisible; if the Cayley graph $\Gamma(G,S)$ is (uniformly) $\epsilon$-coarsely divisible.

The property of being $\epsilon$-coarsely loop divisible depends on $\bigl(\omega, d\bigr)$.  When there is a chance of confusion, we will say that $X$ is $\epsilon$-coarsely loop divisible with respect to $\bigl(\omega, d\bigr)$.

\end{defn}

If $X$ is $\epsilon$-coarsely loop divisible for every $\epsilon$ and the bound $K(\delta, \epsilon)$ can be chosen independent of both $\delta$ and $\epsilon$, then $\con$ has Olshanskii-Sapir's property LDC($K$) as defined in \cite{OS1}.

We will see (\ref{indep}) that for finitely generated groups this definition is independent of the generating set in the sense that if $S,S'$ are two finite generating sets for $G$, then $\Gamma(G,S)$ is $\epsilon$-coarsely loop divisible if and only if $\Gamma(G,S')$ is $\epsilon'$-coarsely loop divisible for some $\epsilon'>0$.

\begin{rmk}\label{iterated}\ulabel{iterated}{Remark}
Suppose that $\vartheta$ is bounded on $\bigl[\frac{n}{2^l}, n\bigr]$ by $K$.  Let $\alpha$ be a loop of length $n$ and fix a partition of $\alpha$ into at most $\vartheta(n)$ pieces with mesh less than $\frac{n}2$.  As in \ref{extend}, the partition can be extended to the $1$-skeleton of the partition such that each loop has length less than $\frac{n}{2}$.  We can then partition each piece with length at least $\frac{n}{4}$ into at most $K$ pieces of length less than $\frac{n}{4}$. This builds a $\frac{n}{4}$-partition of $\alpha$ with at most $K^2$ pieces. Hence $\vartheta^2(n)\leq K^2$.  Iterating this process, we obtain $\vartheta^l(n) \leq K^l$.
\end{rmk}

\begin{lem}\label{lemmaiterated}\ulabel{lemmaiterated}{Lemma}
Fix $l\in\N$.  If $X$ is $\epsilon$-coarsely loop divisible, then for every $\delta\in(0,\epsilon)$ there exists an $A\subset \N$ with $\omega(A)=1$ such that $\vartheta^l$ restricted to $ \bigcup\limits_{n\in A}[\delta d_n, \epsilon d_n]$ is bounded by a
constant $K=K(\delta,\epsilon,l)$.
\end{lem}

Thus the coarse loop division property does not depend on which function $\vartheta^l$ is used in its definition.

\begin{proof}
Suppose $X$ is $\epsilon$-coarsely loop divisible.  Fix $\delta$ such that $0<\delta<\epsilon$.  Choose a $K$ and an $\omega$-large $A$ such that $\vartheta$ restricted to $\bigcup\limits_{n\in A}\bigl[\frac{\delta}{2^{l}} d_n, \epsilon d_n\bigr]$ is bounded by $K$.  By \ref{iterated}, $\vartheta^l$ restricted to $\bigcup\limits_{n\in A}\bigl[\delta d_n, \epsilon d_n\bigr]$ is bounded by $K^{l}$.

\end{proof}

\begin{defn}\label{db}
Let $(\gamma_n)$ be a sequence of loops in a metric space $X$ and $d= (d_n)$ an $\omega$-divergent sequence of real numbers.  Then $(\gamma_n)$ is \emph {not $(m,d, \epsilon,\delta)$-partitionable} if $\delta d_n\leq |\gamma_n|\leq \epsilon d_n$ and $P(\gamma_n,|\gamma_n|/2)> m$ $\omega$-almost surely. When $d$ and $\epsilon$ are fixed, we will say that $(\gamma_n)$ is \emph{not $(\delta,m)$-partitionable}.  Additionally; given a sequence of loops which is not $(\delta,m)$-partitionable, we will say that a fixed member $\gamma_n$ of the sequence is not $(\delta ,m)$-partitionable if $\delta d_n\leq |\gamma_n|\leq \epsilon d_n$ and $P(\gamma_n,|\gamma_n|/2)> m$.

\end{defn}

\begin{rmk}\label{O}\ulabel{O}{Remark} Let $\gamma:\partial \mathbb D\to X$ be parameterized by arc length. Suppose that $8\diam{\gamma}<|\gamma| $.  Let $P$ be the cellular decomposition of the unit disc $\mathbb D$ such that $P^{(1)}$ is $\partial \mathbb D\cup A$ where $A$ is a maximal square inscribed in $\mathbb D$.  Then $\Pi:  P^{(0)}\to X$ defined by $\Pi(t) = \gamma(t)$ is a partition of $\gamma$ with five pieces (four $2$-gons and one $4$-gon) and $\mesh{(\Pi)}\leq\max\{ \frac{|\gamma|}{4}+\diam{\gamma},\ 4\diam{\gamma}\}< \frac{|\gamma|}{2}$.

Thus, if $(\gamma_n)$ is not $(m,d, \epsilon,\delta)$-partitionable for some $m\geq5$, then $|\gamma_n|\leq 8\diam{\gamma_n}$.  Hence; if $(\gamma_n)$ is not $(\delta,m)$-partitionable, then $|\gamma_n| \leq O\bigl(\diam{\gamma_n}\bigr)$ where the big $O$ constant is independent of $(\gamma_n)$.
\end{rmk}

The following two propositions were proved by Papasoglu in \cite[pages
792-793]{pap}.  The formulations are slightly different here but the proofs are the same. The proofs are also outlined in \cite{OS1}.

\begin{prop}\label{pap1}\ulabel{pap1}{Proposition}
Let $X$ be a metric space and $(\gamma_n)$ a sequence of loops in $X$ such that $|\gamma_n| = O(d_n)$. If each $\gamma_n$ has a $\delta_n$-partition with at most $k$ pieces, then the loop $\gamma(t) = \bigl(\gamma_n(t)\bigr)$ in $\con$ has a $\delta$-partition with at most $k$ pieces where $\delta = \lim^\omega \frac{\delta_n}{d_n}$.
\end{prop}

\begin{prop}\label{pap2}\ulabel{pap2}{Proposition}
Let $X$ be a complete geodesic metric space. If $X$ is uniformly $\epsilon$-coarsely loop divisible for every $\epsilon>0$ with respect to the pair $\bigl(\omega, d\bigr)$, then $\con$ is simply connected.
\end{prop}

To prove \ref{pap2}, Papasoglu uses \ref{pap1} to show that every loop in \newline $\con$ is partitionable and the number of pieces is independent of the loop.  He then iterates the process of taking partitions and extending them to the $1$-skeleton as in \ref{extend}.  A consequence of this procedure is that the diameter of the constructed disc is proportional to the length of the loop (the proportionality constant can be chosen to be the bound on the number of pieces in the partitions).

\begin{lem}\label{corb}\ulabel{corb}{Lemma}
Suppose that $X$ is a complete geodesic metric space which is uniformly $\epsilon$-coarsely loop divisible with respect to the pair $\bigl(\omega, d\bigr)$. Then there exists a constant $K$ such that every loop in $\con$ with diameter less than $\frac{\epsilon}{8}$ bounds a disc with diameter less than $K\epsilon$.
\end{lem}

\proof Since $X$ is uniformly $\epsilon$-coarsely loop divisible every loop in $\con$ with length less than $\epsilon$ is partitionable with a uniform bound on the number of pieces required. Suppose that a loop in $\con$ has length at least $\epsilon$ and diameter less than $\frac{\epsilon}{8}$. Then it has a partition with 5 pieces by \ref{O}.  Thus every loop in $\con$ with diameter less than $\frac{\epsilon}{8}$ is partitionable and we can  apply the proof of \ref{pap2}.
\endproof

\ref{corb} can be restated in the following way.

\begin{prop}\label{eldpsimply}\ulabel{eldpsimply}{Proposition}
Let $X$ be a complete geodesic metric space. If $X$ is uniformly $\epsilon$-coarsely loop divisible, then $\con$ is uniformly locally simply connected.
\end{prop}

\begin{prop}\label{simply}\ulabel{simply} {Proposition}
Let $X$ be a complete geodesic metric space. If $X$ is uniformly $\epsilon$-coarsely loop divisible, then $X$ has an asymptotic cone which is simply connected.
\end{prop}

\proof Suppose that $X$ is uniformly $\epsilon$-coarsely loop divisible for some $\bigl(\omega, d\bigr)$ and
$\epsilon>0$. We can consider an ultralimit of the metric spaces $X_k= \conn{d_n/k}$. By Corollary 3.24 in \cite{DS}, $\lim^\omega X_k$ is again an asymptotic
cone of $X$. Thus we can choose $\bigr(\mu,(p_n)\bigl)$ and $(x_n)$ such that $\operatorname{Con}^\mu\bigl(X,(x_n),(p_n)\bigr)$ and $\lim^\omega X_k$ are isometric.

The identity map $id$ from $\con$ to $X_k$ rescales distances by a fixed constant which implies that $P(\gamma, |\gamma|/2) = P(id(\gamma),id(|\gamma|)/2 ) $.  Since $X$ is uniformly $\epsilon$-coarsely loop divisible, there exists $\nu_0$ such that $P(\gamma, |\gamma|/2)<\nu_0$ for every loop $\gamma$ contained in a ball of radius $\frac{\epsilon}4$ in $\con$. Hence; every
loop $\gamma$ contained in a ball of radius $\frac{k\epsilon}{4}$ in $X_k$
has the property that $P(\gamma, |\gamma|/2 )<\nu_0$. Thus for any loop $\alpha$ in
$\lim^\omega X_k$; $P(\alpha, |\alpha|/2 )<\nu_0$.  Hence $\lim^\omega X_k$ is uniformly $\epsilon$-coarsely loop divisible for every $\epsilon>0$ with respect to the pair $\bigl(\mu,(p_n)\bigr)$ and \ref{pap2} implies that $\operatorname{Con}^\mu\bigl(X,(x_n),(p_n)\bigr)$ is simply connected.
\endproof

\ref{corb} shows that a necessary condition for a group to be uniformly $\epsilon$-coarsely loop divisible for every $\epsilon>0$
is that all loops in $\con$ bound discs with diameters proportional
to their length.

\begin{rmk}\label{hawaiianearing}\ulabel{hawaiianearing}{Remark}

Let $X$ be a topological space. The \emph{topological cone} of $X$ written $\hat X$ is the quotient space of $X\times [0,1] $ obtained by identifying all points $(x,1)$ for $x\in X$. $X$ canonically embeds in $\hat X$ by $x\mapsto (x,0)$ and we will generally identify $X$ with $X\times\{0\}$.  The \emph{Hawaiian earring} is the one-point compactification of a sequence of disjoint arcs and can be realized in the plane as the union of circles centered at $(0,\frac1n)$ with radius $\frac1n$.  We will use $\textbf{E}$ to denote this subspace of the plane and $\mathbf{a}_n$ to denote the
circle centered at $(0,\frac1n)$ with radius $\frac1n$.  The \emph{Hawaiian earring group} is $\pi_1(\textbf{E},(0,0)) = \mb H$.  Let $\textbf{E}_n = \bigcup\limits_{i\geq n} \mathbf{a}_i$ and $\mathbb H_n = \pi_1(\textbf E_n,(0,0))\leq \mathbb H$.  Notice that $\textbf{E}_n$ is homeomorphic to $\textbf {E}$ which implies that $\mathbb H_n$ is isomorphic to $\mathbb H$.

$\hat{\mathbf E}$ is a space which is not uniformly $\epsilon$-coarsely loop divisible but is simply connected and not locally simply connected. Suppose that instead of coning from a single point, we were to cone each circle individually.  Then as long as we required that the sequence of cone points converged to the wedge point of $\mathbf E$ but at a rate slower than the radii of the loops, this space would be locally simply connected but not be uniformly $\epsilon$-coarsely loop divisible for any $\epsilon$.  These two examples show that for general metric spaces being uniformly $\epsilon$-coarsely loop divisible is not a necessary condition for a space to be simply connected or locally simply connected.\end{rmk}

Erschler-Osin \cite{EO} and Dru\c tu-Sapir \cite{DS} proved that many metric spaces $\pi_1$-embed into the asymptotic cones of finitely generated groups.  In both papers, the spaces that were $\pi_1$-embedded into the asymptotic cones of finitely generated groups were uniformly locally simply connected.  

A positive answer to either Question \ref{quest1a} or Question \ref{quest1} would imply that  the results of Erschler-Osin and Dru\c tu-Sapir cannot be extended to spaces which are semilocally simply connected but not locally simply connected.

We will now prove some implications of the coarse loop division property.

The following lemma is an immediate consequence of \ref{pap1} and \ref{limitloop}.

\begin{lem}
Suppose that $X$ is a complete geodesic metric space which is $\epsilon$-coarsely loop divisible.  Every loop in $\con$ with length less than $\epsilon$ is partitionable.
\end{lem}

\begin{lem}\label{conpart} \ulabel{conpart}{Lemma}
Suppose that $X$ is a complete homogeneous geodesic metric space.  If every loop in \newline $\con$ with length less than $\epsilon$ is partitionable, then $X$ is $\epsilon'$-coarsely loop divisible with respect to the pair $\bigl(\omega, d\bigr)$ for every $\epsilon'<\epsilon$.
\end{lem}

\begin{proof} Suppose that $X$ is not $\epsilon'$-coarsely loop divisible with respect to the pair $\bigl(\omega, d\bigr)$ for some $\epsilon'$ with $0<\epsilon'<\epsilon$.  Then there exists a $\delta>0$ such that for every $\omega$-large $A$, $\vartheta$ restricted to $ \bigcup\limits_{n\in A}[\delta d_n, \epsilon' d_n]$ is unbounded.

Let $\gamma_n$ be a loop based at $x_n$ such that $\delta d_n\leq|\gamma_n|\leq \epsilon' d_n$ and satisfies at least one of the two following properties.

\begin{enumerate}[a)]
    \item $P\Bigl(\gamma_n, \dfrac{|\gamma_n|}{2}\Bigr)>n$
    \item $P\Bigl(\gamma_n, \dfrac{|\gamma_n|}{2}\Bigr)\geq P\Bigl(\alpha, \dfrac{|\alpha|}{2}\Bigr)$
        for all $\alpha$ such that $\delta d_n\leq|\alpha|\leq \epsilon'
        d_n$
\end{enumerate}

Let $m_n = P\Bigl(\gamma_n, \dfrac{|\gamma_n|}{2}\Bigr)$.  Since $\vartheta$ restricted to $ \bigcup\limits_{n\in A}[\delta d_n, \epsilon d_n]$ is unbounded for every $\omega$-large $A$; $\lim^\omega m_n = +\infty$. Thus for every $m$, $(\gamma_n)$ is not $(\delta ,m)$-partitionable.

The path $\gamma(t)=\bigl(\gamma_n(t)\bigr)$ is a well-defined loop in $\con$ with positive diameter and arc length at most $\epsilon'< \epsilon$.  By assumption, there exists a $\frac{|\gamma|}{2}$-partition of $\gamma$ with $L$ pieces.  However, this induces a $|\gamma_n|/2$-partition of $\gamma_n$ with $L$ pieces $\omega$-almost surely.  Hence $P\Bigl(\gamma_n, \dfrac{|\gamma_n|}{2}\Bigr)\leq L$ $\omega$-almost surely, which contradicts our choice of $m_n$.

\end{proof}

\begin{prop}\label{slsc}\ulabel{slsc}{Proposition}
Suppose that $X$ is a complete homogenous geodesic metric space.  If \newline $\con$
is semilocally simply connected then $X$ is $\epsilon$-coarsely loop divisible for some $\epsilon>0$.
\end{prop}

It is not known whether the converse holds. The converse is Question \ref{quest1} with the \emph{uniform} hypothesis removed.

\begin{proof}[Proof of \ref{slsc}.]

Suppose that every loop  in $\con$ contained in a ball of radius $\epsilon$ is nulhomotopic in $\con$.  Then for every $\gamma$ of length at most $\epsilon$, there exists a continuous map of a disc into $\con$ which extends $\gamma$ and is necessarily uniformly continuous.  For sufficiently small $\nu$, a $\nu$-partition of the disc gives us a finite $|\gamma|/2$-partition for $\gamma$.  Then the result follows from \ref{conpart}

\end{proof}

\begin{thm}\label{uncountable}\ulabel{uncountable}{Theorem}
Let $X$ be a complete homogenous geodesic metric space. If $X$ is not $\epsilon$-coarsely loop divisible with respect to $\bigl(\omega,d\bigr)$ for every $\epsilon>0$, then $\con$ has uncountable fundamental group.
\end{thm}

The proof will require the following result of Cannon and Conner.

\begin{thm}[Cannon, Conner \cite{cc1}]\label{artin}\ulabel{artin}{Theorem}
Let $X$ be a topological space, let $\phi :\pi_1(X,x_0) \to L$ be a
homomorphism to a group $L$, $U_1\supset U_2 \supset\cdots$ be a
countable local basis for $X$ at $x_0$, and $G_i$ be the image of the
natural map from $\pi_1(U_i, x_0)$ into $\pi_1(X,x_0)$. If $L$ is countable,
then the sequence $\phi(G_1)\supset \phi(G_2)\supset\cdots$ is eventually
constant.
\end{thm}

\proof[Proof of \ref{uncountable}]

Let $X$ be a complete homogenous geodesic metric space.  Suppose that $X$ is not $\epsilon$-coarsely loop divisible for any $\epsilon$ and $\con=X^\omega$ has countable fundamental group. Let $i_*$ be the identity map on $\pi_1(X^\omega,e)$. \ref{artin} implies that $i_*(G_n)$ is eventually constant where $G_n$ is the image of the natural map from $\pi_1(B_{1/n}(e), e)$ into $\pi_1(X^\omega,e)$.

Fix $N$ such that this sequence is constant for $m\geq N$, and let $\epsilon = 1/N$. Therefore every loop in $B_{1/N}(\tilde x)$ can be homotoped into $B_{1/m}(e)$ for any $m\geq N$.  In general, this will not imply that the ball is simply connected. However, it does imply that every loop $\gamma$ of length less than $\epsilon$ has a partition with finitely many pieces and mesh at most $\frac{|\gamma|}{2}$.  Then \ref{conpart} implies that $X$ is $\epsilon$-coarsely loop divisible which is a contradiction.

\endproof

\begin{thm}\label{notfree}\ulabel{notfree}{Theorem}
Let $X$ be a complete homogenous geodesic metric space. If $X$ is not $\epsilon$-coarsely loop divisible with respect to the pair $\bigl(\omega,d\bigr)$ for every $\epsilon>0$, then the fundamental group of $\con$ is not free.  In particular, if $\pi_1\Bigl(\con\Bigr) = *_j G_j$ for some free product of groups $G_j$, then there exists a $j$ such that $G_j$ is uncountable and not free.
\end{thm}

We will use the following two results in the proof of \ref{notfree}.

\begin{thm}\label{freeimage}\ulabel{freeimage}{Theorem}
Suppose that $\phi:\mathbb H \to \mathbb F$ is a surjective homomorphism where $\mathbb
F$ is a free group.  Then  $\mathbb F$ has finite rank.
\end{thm}

If we consider homomorphisms from the natural inverse limit containing $\mb H$ to free groups, then this is a theorem of Higman \cite{hig}.  When we consider homomorphism from $\mb H$, this is a consequence of \ref{artin} and a proof can be found in \cite{cs}.

\begin{thm}[\cite{eda6}]\label{conjugate}\ulabel{conjugate}{Theorem}
Suppose that $\phi:\mathbb H \to *_jG_j$ is a homomorphism.  Then there exists an $n$ such that $\phi(\mathbb H_n)$ is contained in a subgroup which is
conjugate to $G_j$ for some $j$.
\end{thm}

\begin{proof}[Proof of \ref{notfree}.]
Since $X$ is not $\epsilon$-coarsely loop divisible with respect to $\bigl(\omega,d\bigr)$ for every $\epsilon$, we may find a null sequence of loops $\alpha_n$ in $\con$ such that $\alpha_n$ has no finite $\frac{|\alpha_n|}{2}$-partition.  Since $\con$ is transitive by isometries, we may choose $\alpha_n$ such that $\alpha_i(0) = \alpha_j(0)= e$ for all $i,j$.  By passing to a subsequence, we may assume that $|\alpha_n|< \frac{|\alpha_{n-1}|}2$.  This implies that the ball of radius $|\alpha_n|$ does not contain a loop which is homotopic to $\alpha_i$ for $i<n$.  Since $\alpha_n$ forms a null sequence of loops and $\alpha_i(0) = \alpha_j(0)$ for all $i,j$, there exists a continuous map $f$ from $\textbf E$ to $\con$ such that $f(\mathbf a_n) = \alpha_n$.

Suppose that $\pi_1\bigl(\con, (x_n)\bigr)$ was free.  Then  $f_*(\mb H)$ would be free and \ref{freeimage} would then imply that it has finite rank.  Hence $f_*(\mb H)$ is countable and \ref{artin} implies that $f_*\Bigl(\pi_1\bigl(\textbf E_n, (0,0)\bigr)\Bigr)$ as a sequence in $n$ is eventually constant which contradicts our choice of $\alpha_n$.

Thus for every $n$, $f_*(\mathbb H_n)$ is uncountable and not free.  The last claim of the theorem follows from \ref{conjugate}.

\end{proof}

\begin{prop}\label{notsimple}\ulabel{notsimple}{Proposition}

Let $X$ be a complete homogenous geodesic metric space. If $X$ is not $\epsilon$-coarsely loop divisible with respect to $\bigl(\omega,d\bigr)$ for every $\epsilon>0$, then the fundamental group of $\con$ is not simple.

\end{prop}

\begin{proof}
Let $X^\omega = \con$ and $\alpha_i$ be a null sequence of loops in $X^\omega$ constructed as in the proof of \ref{notfree}.  Let $A_n$ be the union of the images of $\alpha_i$ for $i> n$.  Let $Y_n$ be the topological cone of $A_n$ in $X^\omega$, i.e. the subset of $\hat {X}^\omega$ consisting of $\con\times \{0\}$ and the canonically embedded $\hat {A}_n$.  The inclusion map $\iota_n: \con \to Y_i$ defined by $x \mapsto (x,0)$ induces a map $\iota_{n*}$ on fundamental groups  with non-trivial kernel.  Hence, it is enough to show that the induced map on fundamental groups is non-trivial.

\begin{clm*}
For $i\leq n$, $\iota_n(\alpha_i)$ is homotopically essential in $Y_n$.
\end{clm*}

\emph{Proof of claim.}
Suppose that $h:\mathbb D \to Y_n$ is a nullhomotopy of $\iota_n(\alpha_i)$ for some $i\leq n$ where $\mathbb D$ is the unit disk in the plane.  Let $z$ be the cone point.  Notice that $A_n$ separates $Y_n$.  Hence the boundary of each component of $h^{-1}(\hat A_n)$ is contained in $h^{-1}(A_n)$.  By possible modifying $h$, we may assume that each component of $h^{-1}(\hat A_n)$  which is not contained in $h^{-1}(A_n)$ intersects the cone point $z$.  (Suppose $B$ is a component of  $h^{-1}(\hat A_n)$ such that $h(B)\cap \{z\} = \emptyset$.  Then we can push $h$ down along cone lines to insure that $h(B)\subset A_n$.)

Since each component of $h^{-1}(\hat A_n)$ which is not contained in $h^{-1}(A_n)$ intersects $h^{-1}(A_n)$ and $h^{-1}(z)$ (two disjoint closed sets), there are only finitely many components of $h^{-1}(\hat A_n)$ which are not contained in $h^{-1}(A_n)$.

Let $C$ be the component of $h^{-1}(X^\omega)$ containing the unit circle in the plane.  Then $C$ is a planar annulus of finite genus.  $\bigl($The genus is equal to the number of components of $h^{-1}(\hat A_n)$ which are not contained in $h^{-1}(A_n)$.$\bigr)$  Since each boundary component of $C$ except the unit circle maps into $A_n$, the diameter of its image is at most $|\alpha_{n+1}|< \frac{|\alpha_n|}{2}$.  This implies that $h: C \to X^\omega$ can be used to find a finite partition of $\alpha_i$ with mesh at most $\frac{|\alpha_n|}{2}$.  Hence, $\alpha_i$ is partitionable which contradicts our choice of $\alpha_i$.

\end{proof}

The property of being $\epsilon$-coarsely loop divisible is a quasi-isometry invariant in the following sense.

\begin{prop}\ulabel{indep}{Proposition} If $X$ and $Y$ are two quasi-isometric homogenous geodesic metric spaces, then  $X$ is $\epsilon$-coarsely loop divisible if and only if $Y$ is $\epsilon'$-coarsely loop divisible for some $\epsilon'>0$.\end{prop}

\begin{proof}
If $X$ and $Y$ are are quasi-isometric, then their cones are bi-lipschitz.  If $X$ is $\epsilon$-coarsely loop divisible for some $\epsilon>0$, then \ref{pap1} implies that every loop of length less than $\epsilon$ in $\con$ is partitionable.

Let $f: \con \to \operatorname{Con}^\omega \bigl(Y,e',d\bigr)$ be a bi-lipschitz map with bi-lipschitz constant $C$.  By iterating partitions as in \ref{iterated}, we can see that every loop of length less than $\epsilon$ in $\con$ has a partition with finitely many pieces and mesh at most $\frac{|\gamma|}{2C}$.  Let $\gamma$ be a loop in $\operatorname{Con}^\omega\bigl(Y,e',d\bigr)$  with length less than $\frac\epsilon C$. Then $f^{-1}\circ\gamma$ has length at most $\epsilon$ and hence has a partition with mesh at most $\frac{|\gamma|}{2C}$.  Then composing the partition with $f$ gives us a partition of $\gamma$ with finitely many pieces and mesh at most $\frac{|\gamma|}{2}$.
\ref{conpart} implies that $Y$ is $\epsilon'$-coarsely loop divisible for every $\epsilon'< \frac\epsilon C$.

\end{proof}

\subsection{Absolutely non-divisible sequences}\label{subsection absolutely non-divisible}

\begin{defn}
A sequence of loops $(\alpha_n)$ is  \emph{absolutely non-divisible} if there exists an $M$ such that the sequences $P\Bigl(\alpha_n,
\dfrac{|\alpha_n|}{M}\Bigr)$ and $|\alpha_n|$ both tend to $+\infty$ and
$\Bigl\{\dfrac{|\alpha_{n+1}|}{|\alpha_{n}|}\Bigr\}$ is bounded.
\end{defn}

\begin{rmk}\label{absulotely non-divisible}\ulabel{absulotely non-divisible}{Remark}
Suppose that $|\alpha_n|$ is unbounded and $\Bigl\{\dfrac{|\alpha_{n+1}|}{|\alpha_{n}|}\Bigr\}$ is bounded.  To simplify our notation, we will let $|\alpha_n|= a_n$ and $B$ be a bound on $\Bigl\{\dfrac{a_{n+1}}{a_n}\Bigr\}$.

Let $n_0 = 1$.  Then  we can define $\{n_i\}$, inductively, by letting $n_{i+1} = \min\{n\in\mathbb N \ | \ a_n> a_{n_i}+1 \text{ and } n> n_{i}\}$.  If $n_{i+1}\not= n_{i}+1$, then $a_k\leq a_{n_{i}}+1$ for all $n_{i}\leq k<n_{i+1}$.

Thus $\dfrac{a_{n_{i+1}}}{a_{n_i}} = \dfrac{a_{n_{i+1}}}{a_{(n_{i+1})-1}}\cdot\dfrac{a_{(n_{i+1})-1}}{a_{n_i}}\leq B\dfrac{a_{n_{i}}+1}{a_{n_{i}}}\leq B\max\bigl\{2,\frac{2}{a_{n_0}}\bigr\}$.

Therefore $\{a_{n_i}\}_i$ is a subsequence which is absolutely
non-divisible.

Thus, it is possible to loosen this definition slightly and only require that $|\alpha_n|$ be unbounded.
\end{rmk}

\begin{lem}\label{limit}\ulabel{limit}{Lemma}

Fix $\omega$ an ultrafilter on $\mb N$, $d$ an $\omega$-divergent sequence, and $A$ an infinite subset of the natural
numbers. Suppose that $A = \{b_1<b_2<b_3<\cdots\}$ has the property that the set of ratios $\bigl\{\frac{b_{k+1}}{b_k}\bigr\}$ is bounded by $L$. Then for any $\epsilon >0$, there exists a sequence $(a_n)$ in $A$ such that $\lim^\omega \frac{a_n}{d_n} \in \bigl[\frac\epsilon L,\epsilon\bigr]$.
\end{lem}

We allow $a_n$ to have repeated terms; hence, $a_n$ is not necessarily a subsequence of
$b_n$. However $a_n$ is not eventually constant, since $\lim^\omega d_n = +\infty$.

\begin{proof}

Let $L$ be an upper bound on the set $\bigl\{\frac{b_{n+1}}{b_n}\bigr\}$. For all $n$ such that $\frac{b_1}{d_n}\leq \epsilon$, choose $(i_n)$ such that $ \frac{b_{i_n}}{d_n}\leq \epsilon< \frac{b_{i_n+1}}{d_n}$.   Let $a_n = b_{i_n}$.

Then $\epsilon d_n< b_{i_n+1}$ which implies that $\frac\epsilon L<\frac{b_{i_n}}{d_{n}} = \frac{a_n}{d_n}\leq \epsilon$. For all $n$ such that $\frac{b_1}{d_n}> \epsilon$, let $a_n = b_1$.  Then $\lim^\omega \frac{a_n}{d_n} \in \bigl[\frac\epsilon L,\epsilon\bigr]$.
\end{proof}

\begin{lem}\label{ndivisible}\ulabel{ndivisible}{Lemma}
Let $X$ be a complete geodesic metric space.  If there exists a sequence of absolutely non-divisible loops in $X$, then for every pair $\bigl(\omega, d\bigr)$ and $\epsilon>0$, $X$ is not $\epsilon$-coarsely loop divisible.
\end{lem}

\begin{proof}
Fix $\epsilon>0$, $\omega$ an ultrafilter, and $d$ an $\omega$-divergent sequence of real numbers.

Let $(\gamma_n)$ be a sequence of loops in $X$ which is absolutely non-divisible.
By passing to a subsequence as in \ref{absulotely non-divisible}, we may assume that the lengths of $\gamma_n$ are nondecreasing.  Let $A =\{|\gamma_n|\}$ and $L$ be an upper bound on $\bigl\{\frac{|\gamma_{n+1}|}{|\gamma_n|}\bigr\}$.

Let $(a_n)\subset A$ be a sequence constructed as in \ref{limit} where we replace $\epsilon $ by $\frac\epsilon2$. Consider
the sequence of loops $\gamma_{k_n}$ where $\gamma_{k_n}$ has length $a_n$.  Since $\lim^\omega \frac{a_n}{d_n} \in \bigl[\frac{\epsilon} {2L},\frac\epsilon2 \bigr]$, we have $|\gamma_{k_n}| \in\bigl[\frac{d_n\epsilon}{L},d_n\epsilon\bigr]$ $\omega$-almost surely.   However,
$P\bigl(\gamma_{k_n}, \frac{|\gamma_{k_n}|}{M}\bigr)$ tends to $+\infty$. Hence, $\vartheta$ restricted to $ \bigcup\limits_{n\in A}[\frac\epsilon {L} d_n, \epsilon d_n]$ is unbounded for all $\omega$-large $A$.  Hence \ref{lemmaiterated} implies that $X$ is not $\epsilon$-coarsely loop divisible.  Since $\epsilon$ was arbitrary, $X$ is not $\epsilon$-coarsely loop divisible with respect to $\bigl(\omega, d\bigr)$ for any $\epsilon>0$.  Since $\bigr(\omega, d\bigl)$ were also arbitrary, this completes the proof.
\end{proof}

\ref{ndivisible} and \ref{uncountable} immediately imply the following corollary.

\begin{cor}\label{absuncountable}\ulabel{absuncountable}{Corollary}
Let $X$ be a complete homogenous geodesic metric space.  If there exists a sequence
of loops in $X$ which is absolutely non-divisible, then every asymptotic cone of $X$ has uncountable fundamental group and is not semi-locally simply connected at any point.
\end{cor}

\subsection{Simply connected cones}\label{section simply connected cones}
When Papasoglu proved \ref{pap2}, he used the uniform bound on the number of pieces in a partition to construct discs.  Being coarsely loop divisible implies that loops in the cone are partitionable but does not give a bound on the number of pieces which is independent of the loop.  Thus Papasoglu's method is insufficient to build discs when a space is only coarsely loop divisible and not uniformly coarsely loop divisible.  Here we will show that requiring a linear isodiametric function on partitions along with coarsely loop divisible is sufficient to build discs.

When considering subsets of $\mathbb N$, we will write $[a,b]$ for the set $\bigl\{n\in\mathbb Z \ | \ a\leq n\leq b\bigr\}$.  For $A\subset \mathbb N$, we will let $A^c =\mathbb N\backslash A$.  For $d\in \mathbb R^+$ and $A\subset \mathbb N$, let $\mathcal M_d(A)= \bigl\{x\in N \ |  \ [\frac{x}{d}, xd]\cap A \not=\emptyset\bigr\}$.

\begin{prop}\label{gldp for all}\ulabel{gldp for all}{Proposition}
Suppose that for every $\bigl(\mu, (p_n)\bigr)$ there exists an $\epsilon>0$ such that $X$ is $\epsilon$-coarsely loop divisible with respect to the pair $\bigl(\mu, (p_n)\bigr)$.  Then there exists a pair $\bigl(\omega, d\bigr)$ such that $X$ is $\epsilon$-coarsely loop divisible for every $\epsilon>0$ with respect to $\bigl(\omega, d\bigr)$.
\end{prop}

Before we can prove \ref{gldp for all}, we will need a necessary condition for $X$ to be $\epsilon$-coarsely loop divisible for every pair $\bigl(\omega, d\bigr)$.

\begin{lem}\label{allways gldp}
 Let $A_k = \vartheta ^{-1}\bigl([1,k]\bigr)$, $A_k' =\vartheta ^{-1}\bigl(\{k\}\bigr)$, and $B_k = \vartheta ^{-1}\bigl([k+1,\infty)\bigr)$. If  $X$ is $\epsilon$-coarsely loop divisible for every pair $\bigl(\omega, d\bigr)$, then for every $s\in\mathbb N$ there exists $b=b(s)$ such that

\begin{enumerate}[i)]
    \item\label{growing} if $c_k^i = \sup\Bigl\{ \frac yx | \ i<x \text{ and } [x,y] \subset A_k\Bigr\}$, then $c_k =\lim\limits_{i\to\infty} c_k^i$ and $c_k \to \infty$,

    \item\label{boundedcomplement} if $b_k = \sup\Bigl\{ \frac yx | [x,y] \subset \mathcal M_s(B_k)\Bigr\}$, then   $b_k<b$ for all sufficiently large $k$, and

     \item if $c_k' = \Bigl\{ \frac yx | [x,y] \subset \mathcal M_s(A'_k)\Bigr\} $, then $c_k'<\alpha$ for all $k$.

\end{enumerate}
\end{lem}

\begin{proof}For fixed $k$, $c_k^i$ is a decreasing sequence in $i$.  Hence, $c_k$ exists as an extended real number ($c_k^i$ might be infinite for all $i$). The sequence $c_k$ is increasing since the sets $A_k$ are nested.

\emph{Proof of (\ref{growing}).} Suppose that there existed $L$ such that $c_k< L$ for all $k$.  We may choose an increasing sequence $k_n$ such that $c_n^i< 2L$ for all $i>{k_n}$.  Thus for every interval $[x,y]$ such that $k_n<x$ and $\frac xy \geq 2L$, $[x,y]\not\subset A_n$, i.e. $[x,y]\cap B_n \not=\emptyset$.

Fix an ultrafilter $\omega$ and let $d_n = (k_n)^2$. Suppose $X$ is $\epsilon$-coarsely loop divisible for some $\epsilon>0$.  Then $\bigcup\limits_{n\in A}  [\frac{\epsilon d_n}{2L}, \epsilon d_n]\subset A_t$ for some $t$ and $\omega$-large $A$.  However; for all sufficiently large $n$, $k_n<\frac{\epsilon d_n}{2L}$ which implies that $ [\frac{\epsilon d_n}{2L}, \epsilon d_n]\cap B_n\neq\emptyset$ for all sufficiently large $n$.  This contradictions our choice of $t$ such that $\bigcup\limits_{n\in A}  [\frac{\epsilon d_n}{2L}, \epsilon d_n]\subset A_t$.

\emph{Proof of (\ref{boundedcomplement}).} Fix $s\in\mathbb N$.  Suppose that $(\ref{boundedcomplement})$ does not hold.  Then there exists $[x_n, y_n]\subset \mathcal M_s(B_n)$ such that $\frac{y_n}{x_n}>n$.  Fix an ultrafilter $\omega$ and let $d_n = (x_ny_n)^\frac12$, the geometric center of the interval $[x_n, y_n]$. Then for every $n'\leq n$, $\mathcal M_{\sqrt n}(d_n) \subset [x_n, y_n]\subset \mathcal M_s(B_{n'})$.  (The first inclusion follows by our choice of $d_n$ and the second holds since $\mathcal M_s(B_{n}) \subset \mathcal M_s(B_{n'}) $ for $n'\leq n$.)

Suppose $X$ is $\epsilon$-coarsely loop divisible with respect to $\bigl(\omega, d\bigr)$ for some $\epsilon\in(0,1)$.  For any $0<\delta<\epsilon$, $\bigcup\limits_{n>m}  [\delta d_n, \epsilon d_n] \subset \mathcal M_s(B_m)$ for every $m>\frac1{\delta^2}$.  If $\delta<\frac{\epsilon}{2s}$ and $ [\delta d_n, \epsilon d_n] \subset \mathcal M_s(B_m)$, then $ [\delta d_n, \epsilon d_n] \cap B_m\not=\emptyset$.  Since this hold for every sufficiently large $m$, we can derive a contradiction as in (\ref{growing}).

The proof of (\emph{iii}) is the same as proof of (\emph{\ref{boundedcomplement}}).

\end{proof}

\begin{proof}[Proof of \ref{gldp for all}]We will use the notation from Lemma \ref{allways gldp}.  The lemma is trivial if some $c_k = \infty$.  Thus we will assume that for every $k$, $c_k<\infty$.

Let $k_1' = 1$ and $s_1 = \frac{c_{k_1'}}{3}$.  We may choose $k_1>k_1'$ and $b_{1}$ such that $\sup\Bigl\{ \frac ba \bigl| [a,b] \subset \mathcal M_{s_1}(B_k)\Bigr\}< b_1$ for all $k\geq k_1$.

Suppose that we have inductively define $s_i$, $k_i'$, $k_i$ and $b_{i}$ for all $i<n$.

Choose $k_n'\in \mathbb N$ such that $c_{k_n'}> b_{n-1}^3\cdot c_{k_{n-1}'}^2$ and let $s_n = \frac{c_{k_n'}}{3}$. Again, we may choose $k_n>k_n'$ and $b_{n}$ such that $\sup\Bigl\{ \frac ba | [a,b] \subset \mathcal M_{s_{n}}(B_k)\Bigr\}< b_{n}$ for all $k\geq k_n$.

Choose $[a_{1,1},b_{1,1}]$ a maximal interval in $A_{k_1}$ containing a point of $\bigr(\mathcal M_{s_1}(B_{k_1})\bigl)^c$.  Suppose that for all $i<n$, we have chosen $[a_{i,i},b_{i,i}]$.

Let $[a_{n,n},b_{n,n}]$ be a maximal interval in $A_{k_n}$ containing a point of $\bigr(\mathcal M_{s_n}(B_{k_n})\bigl)^c$ such that $b_{n-1,n-1} < a_{n,n}$.

\begin{clm*}
Let $x\in \bigr(\mathcal M_{s_i}(B_{k_i})\bigl)^c$.  Then $[\frac{x}{s_i},xs_i]\subset A_{k_i}$ and there exist $x'\in \bigr(\mathcal M_{s_{i-1}}(B_{k_{i-1}})\bigl)^c\cap [\frac{x}{s_i},xs_i]$ such that $\mathcal M_{s_{i-1}}\bigl([\frac{x'}{s_{i-1}}, x's_{i-1}]\bigr)\subset [\frac{x}{s_i},xs_i]$.
\end{clm*}

\emph{Proof of claim.}
Let $x\in \bigr(\mathcal M_{s_i}(B_{k_i})\bigl)^c$.  Then $\mathcal M_{s_i}(x)\cap B_{k_i} =\emptyset$ which implies that $[\frac{x}{s_i},xs_i]$ in $A_{k_i}$.

Let $a = \frac{x}{s_i}$ and $ b = xs_i$. Then $\frac xa, \frac bx = s_i= \frac{c_{k_i'}}{3}$.  This implies that $\frac ba >\bigl(\frac{c_{k_i'}}{3}\bigr)^2>\bigl( \frac{b_{i-1}^9\cdot c_{k_{i-1}'}^4}{9}\bigr)$.

Let $t= \max \{b_{i-1}, c_{k_{i-1}'}\}$.  Since $\frac{b}{t^4 a} > b_{i-1}$, $[at^2, \frac {b^2}{t^2}]$ contains a point $x'\in \bigr(\mathcal M_{s_{i-1}}(B_{k_{i-1}})\bigl)^c$. Then the inequality $t\geq c_{k_{i-1}'} >s_{i-1}$, along with the inclusion $\mathcal M_t\bigl( [\frac{x'}{t}, x't]\bigr)\subset[a,b]$ imply that $\mathcal M_{s_{i-1}}\bigl([\frac{x'}{s_{i-1}}, x's_{i-1}]\bigr)\subset [a,b]$.  This completes the proof of the claim.

Fix $n$.  The claim shows that we can find a nested sequence of intervals $[a_{1,n},b_{1,n}]\subset [a_{2,n},b_{2,n}] \subset \cdots\subset [a_{n,n},b_{n,n}]$ such that $[a_{1,n},b_{1,n}]\subset A_{k_i}$ and $\mathcal M_{s_{i-1}}\bigl([a_{i,n}, b_{i,n}]\bigr)\subset [a_{i+1,n}, b_{i+1,n}]$.

Let $d_n = (a_{1,n}b_{1,n})^\frac12$.  Then $\bigcup\limits_{n>i}[\dfrac{d_n}{s_i}, s_id_n] \subset A_{k_n}$.  Therefore $X$ is $\epsilon$-coarsely loop divisible with respect to the pair $\bigl(\omega, d\bigr)$ for all $\epsilon>0$, since $s_i$ diverges.
\end{proof}

This gives us the following analogue to \ref{pap2}.  Rather than require a bound on the number of pieces in a partition, we only require a linear bound on the diameter of partitions and $\epsilon$-coarsely loop divisible for all $\epsilon>0$.

\begin{prop}\label{simply}\ulabel{simply}{Proposition}
Suppose that for some fixed pair $\bigl(\omega, d\bigr)$, a complete geodesic metric space $X$ is $\epsilon$-coarsely loop divisible for all $\epsilon>0$. If there exists an $l, L, N$ and an increasing function $f:\N \to\N$ such that every loop $\gamma$ in $X$ with $|\gamma|\geq L$ has a partition $\Pi$ of $\gamma$ with the property that
    \begin{enumerate}[(i)]
        \item $\Pi$ has at most $f\circ\vartheta^l(|\gamma|)$ pieces,
        \item $\Pi$ is a $\frac{|\gamma|}{2}$-partition of $\gamma$, and
        \item $\diam{\Pi}\leq N|\gamma|$
    \end{enumerate}

then $\con$ is simply connected.
\end{prop}

\begin{proof}
Suppose that for some fixed pair $\bigl(\omega, d\bigr)$ and all $\epsilon>0$, $X$ is $\epsilon$-coarsely loop divisible. Fix $l,L,N$ and $f:\R\to\R$ as in statement of the lemma.

We will break the proof into two parts.  First we will show that every geodesic $n$-gon $\alpha$ in $\con$ which is a limit of geodesic $n$-gons from $X$ bounds a disc of diameter at most $2N|\alpha|$.  We will then show that this is enough to imply that all loops are nullhomotopic.

\textbf{Step 1.}
Let $\alpha$ be a geodesic $n$-gon in $\con$ such that $\alpha(t) = \bigl(\alpha_n(t)\bigr)$ where $\alpha_n$ is a geodesic $n$-gon in $X$.  By hypothesis; for each $n$ such that $|\alpha_n|>L$, there exists a partition $\Pi_n$ of $\alpha_n$ which satisfy conditions $(\emph{i})-(\emph{iii})$ of the lemma.

By \ref{lemmaiterated}, there exists a $K$ and an $\omega$-large set $A$ such that $(\vartheta^l)^{-1}(\bigcup\limits_{n\in A}\bigl[|\frac{|\alpha| d_n}{2},2|\alpha| d_n\bigr])$ is bounded by $K$.  We will assume that for all $n\in A$, $|\alpha_n|\in \bigcup\limits_{n\in A}\bigr[\frac{|\alpha| d_n}{2},2|\alpha| d_n\bigl]$.  Thus $\Pi_n$ has at most $f(K)$ pieces $\omega$-almost surely.

Then \ref{pap1} implies that the partitions $\Pi_n$ induce a partition $\Pi$ of $\alpha$ which satisfies conditions $(1)$ and $(2)$ of the lemma.  In Papasoglu's proof of \ref{pap1}, $\Pi$ is just the $\omega$-limit of the partitions $\Pi_n$; thus, condition $(3)$ is also satisfied for $\Pi$.

Fix $\gamma$ a geodesic $n$-gon in $\con$  such that  $\gamma(t) = \bigl(\gamma_n(t)\bigr)$ for $\gamma_n$ a geodesic $n$-gon in $X$.

We have shown that there exists a partition $\Pi_1: P_1^{(0)}\to \con$ of $\gamma$ into pieces of length $|\gamma|/2$ with the diameter of the partition no greater than $N|\gamma|$.

Proceeding by induction, suppose that we have defined $\Pi_{k}: P_k^{(0)}\to \con$ a partition of $\gamma$ into pieces of length $\frac{|\gamma|}{2^{k}}$ for $k<i$ such that for all $1< k\leq i-1$

\begin{itemize}
    \item $\Pi_{k}$ extends $\Pi_{k-1}$
    \item for $x\in\im{\Pi_k}$ $\d(x, \im{\Pi_{k-1}})\leq \dfrac{N|\gamma|}{2^{k}}$.
\end{itemize}

The partition $\Pi_{i-1}$ extends to a map $\widetilde \Pi_{i-1}$ on the $(1)$-skeleton of $P_{i-1}$ as in \ref{extend}.  Then we can partition each of the subloops into pieces of length less than $|\gamma|/2^{i}$ with the desired diameters.  We can then use these partitions to extend $\Pi_{i-1}$ to $\Pi_i$ satisfying the two induction hypothesis.

For all $i>j$; if $x\in\im{\Pi_i}$, then $\d(x, \im{\Pi_{j}})\leq \sum\limits_{s=j}^i\dfrac{N|\gamma|}{2^{s}}$.  Hence, $\Pi_i$ converges to a continuous function from the unit disc into $\con$ which extends $\gamma$.  Therefore $\gamma$ bounds a disc of diameter $2N|\gamma|$.  This completes Step 1.

\textbf{Step 2.}
Let $Q_n$ be the convex hull of the regular $2^n$-gon inscribed in $S^1$, the unit circle in the plane with the standard Euclidean metric.  Then $Q_n$ has a natural cell decomposition with $2^n$ vertices and $2^n$ edges and one 2-cell.  Furthermore, we may assume that the $0$-skeleton of $Q_n$ form a nested sequence of subsets of $S^1$. Let $A_1^2= Q_2$ which has diameter $2$.  For $n>2$,  $Q_{n}\backslash (\operatorname{interior}(Q_{n-1}))$ is a set of $2^n$ triangles with vertices on $S^1$ each of which share a unique edge with $Q_{n-1}$ and have diameter less than $\frac{\pi}{2^{n-1}}$.  Let $\{A_i^n\}_{i=1}^{2^n}$ be this set of triangles.  Then $A= \bigcup\limits_{i,n} A_i^n$ covers the interior of the unit disc and a dense subset of its boundary.

Fix a loop $\gamma:S^1 \to \con$.

We may choose a geodesic $4$-gon $\gamma_1^2: \partial A_2^{(1)} \to \con$ such that $\gamma_1^2|_{Q_2\cap S^1} = \gamma |_{Q_2\cap S^1}$ and $\gamma_1^2$ is the limit of geodesic $4$-gons from $X$.  We can inductively define geodesic $3$-gons $\{\gamma_i^n: \partial A_i^n\to \con\}$

\begin{enumerate}[i)]
    \item \label{welldefined}$\gamma_i^n|_{\partial A_i^n\cap Q_{n-1}}= \gamma^{n-1}_j|_{\partial A_i^n\cap Q_{n-1}}$ for some $j$ and
    \item \label{restrict}$\gamma_i^n|_{\partial A_i^n\cap S^1}= \gamma|_{\partial A_i^n\cap S^1}$.
\end{enumerate}

Using Step 1,  we can define maps $\{h_i^n: A_i^n \to\con\}$ such that $h_i^n$ is a nullhomotopy of  $\gamma_i^n$ and $\diam{h_i^n}$ no greater than $2N\delta_i^n$ where $\delta_i^n$ is the sum of the distances between the image  of adjacent vertices of $A_i^n$.

This defines a function $h: A\to \con$ by $h(a) = h_i^n(a)$ for some $i$ and $n$.  This is well defined by Condition (\ref{welldefined}).  Since $\gamma$ is continuous on a compact set; for ever $\epsilon>0$, there exists a $K$ such that $2N\delta_i^n< \epsilon$ for all $n>K$.  Thus $h$ is continuous on $A$.  By Condition (\ref{restrict}), $h|_{A\cap S^1} = \gamma|_{A\cap S^1}$ which implies that $h$ extends to a nullhomotopy of $\gamma$.

\end{proof}

\begin{cor}\label{simply2}\ulabel{simply2}{Corollary}
Let $G$ be a group and $S$ a finite generating set for $G$.  Suppose that there exists an $l, L, N$ and an increasing function $f:\N \to\N$ such that every loop $\gamma$ in $\Gamma(G,S)$ with $|\gamma|\geq L$ has a partition $\Pi$ of with the property that
    \begin{enumerate}[(i)]
        \item $\Pi$ has at most $f\circ\vartheta^l(|\gamma|)$ pieces,
        \item $\Pi$ is a $\frac{|\gamma|}{2}$-partition of $\gamma$, and
        \item $\diam{\Pi}\leq N|\gamma|$.
    \end{enumerate}
Then at least one of the following occurs.

\begin{enumerate}[(A)]
    \item $G$ has an asymptotic cone which is not semilocally simply connected and has an uncountable fundamental group.
    \item Every asymptotic cone of $G$ is locally simply connected and $G$ has an asymptotic cone which is simply connected.
\end{enumerate}

\end{cor}

\begin{proof}
If for some ultrafilter and scaling sequence $G$ is not $\epsilon$-coarsely divisible for every $\epsilon>0$, then $G$ has an asymptotic cone which is not semilocally simply connected and has uncountable fundamental group.

Otherwise, for every pair $\bigl(\omega, d\bigr)$; $G$ is $\epsilon$-coarsely divisible with respect to $\bigl(\omega, d\bigr)$ for some $\epsilon>0$.  The proof of \ref{simply} implies that every asymptotic cone of $G$ is locally simply connected.

\ref{gldp for all} implies that there exists a pair $\bigl(\omega, d\bigr)$ such that $G$ is $\epsilon$-coarsely divisible for every $\epsilon>0$ with respect to $\bigl(\omega, d\bigr)$.  \ref{simply} implies that $\gcon$ is simply connected.
\end{proof}

\section{Examples}\label{section examples}

\begin{lem}\label{polynomial length}\ulabel{polynomial length}{Lemma}
Suppose that $G$ is a group with a finite presentation $\pres{S}{R}$ which has an exponential isoperimetric function.  If there exists a sequence of loops $\gamma_n$ in $\Gamma(G,S)$ such that $|\gamma_n|$ grows at most linearly and $\area{\gamma_n}$ has an exponential lower bound, then there exists a sequence of absolutely non-divisible loops in $\Gamma(G,S)$.
\end{lem}

\begin{proof}
Let $\gamma_n$ be a sequence of loops in $\Gamma(G,S)$ such that $|\gamma_n|\leq Ln$ and $Ab^n\leq \area{\gamma_n} = \delta(|\gamma_n|)\leq Dc^{Ln}$ where $\delta$ is the Dehn function for the presentation $\pres{S}{R}$ and $A,b,c,D, L$ are positive constants.  Fix $M$ such that $c^{\frac LM}\leq b$.

Suppose that for some subsequence $n_i$, $\gamma_{n_i}$ has a $\frac{|\gamma_{n_i}|}{M}$-partition with at most $K$ pieces where $K$ is independent of $i$.  Then

\begin{align*}
Ab^{n_i} \leq \area{\gamma_{n_i}} &= \delta(|\gamma_{n_i}|)\leq K\delta(\frac{|\gamma_{n_i}|}{M}) \leq K Dc^{\frac{Ln_i}{M}}.
\end{align*}

This implies that $\dfrac{b^{n_i}}{c^{\frac{Ln_i}M}} = \Bigl(\dfrac{b}{c^{\frac LM}}\Bigr)^{n_i}$ is bounded independent of $i$  which contradicts our choice of $M$.

Thus $P\bigl(\gamma_n, \frac{|\gamma_n|}{M}\bigr)$ diverges and the lemma follows from \ref{absulotely non-divisible}.

\end{proof}

\begin{cor}\label{examples}\ulabel{examples}{Corollary}
Every asymptotic cone of the following groups is not semilocally simply connected and has an uncountable fundamental group which is not free and not simple.  In addition, any decomposition of the fundamental group of an asymptotic cone of one of the following groups as a free product has a factor which is not free and uncountable.

\begin{enumerate}
    \item\label{1} $SL_3(\mathbb Z)$;
    \item\label{2} Baumslag-Solitar groups -- $BS_{pq} = \pres{a,t}{t^{-1}a^pt = a^q}$ for $|p|\neq |q|$;
    \item\label{2i} the $3$-manifold $Sol_3$, $\mathbb R^3$ endowed with the Riemannian metric $ds^2 = e^{2z}dx^2+ e^{-2z}dy^2+dz^2$;
    \item\label{2ii} any extension of $\mathbb R^n$ by $\mathbb R$ via a matrix with all real eigenvalues of norm strictly greater than $1$ and at least two eigenvalues with different sign;
    \item\label{3} Baumslag-Gertsen group -- $\pres{a,t}{(t^{-1}a^{-1} t) a (t^{-1}a t) = a^2}$;
    \item\label{5} $Out(F_n)$ and $Aut(F_n)$ for $n\geq3$;
    \item\label{4} $G_{\ref{4}}=\pres{a,s,t}{[a,a^t]= [s,t]= 1, a a^t = a^s}$; and
    \item\label{6} $G_{\ref{6}} = \pres{\theta_1,\theta_2,a,k}{a^{\theta_i}= a, k^{\theta_i}= ka, i = 1,2}$.

\end{enumerate}
\end{cor}
$G_{\ref{4}}$ is of interest since it is metabelian and not polycyclic. It is sometimes referred to as the Baumslag group. $G_{\ref{6}}$ was constructed by Olshanskii and Sapir and has cubic Dehn function and linear isodiametric function.

\begin{proof}
Epstein and Thurson in \cite{ECT} showed the existence of a sequence of loops in $SL_3(\mathbb R)$ and $BS_{pq}$ for $|p|\neq |q|$ satisfying the conditions of \ref{polynomial length}.

The result for Baumslag-Solitar groups  and $Sol_3$ was already known and is due to \cite{Burillopaper}.  Cornulier showed the existence of a sequence of absolutely non-divisible loops for groups of the type (\ref{2ii}) in \cite{Corn}.

Kassabov and Riley in \cite{KRpreprint} showed that the loops in the Cayley graph of $G_{\ref{4}}$ with label $[a,a^{t^n}]$ have the desired properties.

For $Out(F_n)$ and $Aut(F_n)$ Bridson and Vogtmann exhibit the necessary sequence in \cite{BVpreprint}.

That leaves only (\ref{3}) and (\ref{6}).  Since $G_{\ref{3}} = \pres{a,t}{(t^{-1}a^{-1} t) a (t^{-1}a t) = a^2}$ has a Dehn function which is greater than any tower of exponentials, we cannot apply \ref{polynomial length}.  Gersten in \cite{Ger} showed the existence of a sequence of loops $\gamma_{k}$ such that $\gamma_{k}$  has length $ 3\cdot2^{k+1}$ and area at least $2^{2^{\left.2^{\ldots^{2} }\right\} k~times}}$.  Platonov in \cite{Plat} showed that $ \delta(n) = 2^{2^{\left.2^{\ldots^{2} }\right\} \log_2(n)~times}}$ is an isoperimetric function for $G_{\ref{3}}$.

Suppose that for some subsequence  ${n_i}$, $P\bigl(\gamma_{n_i},\frac{|\gamma_{n_i}|}{6}\bigr)\leq K$.  Then for all $n_i$

\begin{align*}
2^{2^{\left.2^{\ldots^{2} }\right\} n_i~times}} \leq \area{\gamma_{n_i}} &\leq K 2^{2^{\left.2^{\ldots^{2} }\right\} \log_2(\frac{32^{n_i}}{6})~times}}= K 2^{2^{\left.2^{\ldots^{2} }\right\} (n_i-1)~times}}
\end{align*}

which is a contradiction.  Hence $\gamma_k$ is an absolutely non-divisible sequence of loops and the result follows from \ref{absuncountable}.

Ol'shanskii and Sapir in \cite{OS1} constructed a sequence of loops $\gamma_n$ in the Cayley complex of $G_{\ref{6}} = \pres{\theta_1,\theta_2,a,k}{a^{\theta_i}= a, k^{\theta_i}= ka, i = 1,2}$ such that $\gamma_n$ has length $6n$. Additionally, they showed that $\gamma_n$ cannot bound a disc decomposed into at most $l$ subdiscs of perimeter $n$ where $l\leq \sqrt n$ and hence is an absolutely non-divisible sequence of loops.

\end{proof}

\begin{rmk}\label{one-dimensional}\ulabel{one-dimesional}{Remark}
Suppose that $X$ is a one-dimensional metric space and $Y$ is the support of any finite set of paths in $X$.  Then $Y$ is a compact one-dimensional metric space and thus its fundamental group is locally free, residually free, and residually finite by Theorem 5.11 in \cite{cc3}.  Since $X$ is one-dimensional, the homomorphism from the fundamental group of $Y$ to the fundamental group of $X$ induced by set inclusion is injective (see Theorem 3.7 in \cite{cc3}).  Thus $\pi_1(X, x_0)$ is locally free.

Burillo in \cite{Burillopaper} shows that all asymptotic cones of solvable $BS_{p,q}$ with $|p|\neq |q|$ and $Sol_3$ have topological dimension 1.  As well, any extension of $\mathbb R^n$ by $\mathbb R$ via a matrix with all real eigenvalues of norm strictly greater than $1$ and at least two eigenvalues with different sign will have one-dimensional asymptotic cones (see \cite{Corn}).  Thus the fundamental groups of their cones are locally free.

\end{rmk}

To prove \ref{examples}, we analysed mappings of Hawaiian earrings into asymptotic cones and showed that the induced homomorphism's image had the desired properties.  However, this method doesn't give us much information concerning the structure of the rest of the fundamental group.  When an asymptotic cones of a group is one-dimensional, one can apply standard techniques for one-dimensional space, as in Remark \ref{one-dimensional}, to better understand the structure of the fundamental.  Requiring a dimension constraint on the asymptotic cone is a strong condition and does not apply to many well studied groups.

\subsection{An example of a group with locally simply connected cones which is not simply connected}

In \cite{O1}, it was shown that there exists a group with the following properties.

\begin{thm}
There is a finitely generated group $G$ whose Dehn function $f(n)$ satisfies the following properties:
\begin{enumerate}

\item there are sequences of positive numbers $d_i\to \infty$ and $\lambda_i\to \infty$ such that $f(n) \leq cn^2$ for arbitrary integer $n\in[\frac{d_i}{\lambda_i}, d_i\lambda_i]$ and some constant $c$ and

\item there is a positive constant $c'$ and an increasing sequence of numbers $n_i\to \infty$ such that $\frac{f(n_i)}{n^2_i}\to\infty$ but for every $i$, and for every integer $n$ with $n<c'n_i$, we have $f(n) \leq c'n_i^2$.
\end{enumerate}
\end{thm}

Ol'shanskii construct $G$ as a multiple HNN extension of a free group using $S$-machines.

\begin{cor}
\begin{enumerate}[(A)]\hspace{5in}
    \item There exists an $\omega$ such that $\operatorname{Con}^\omega(G,(n_i))$ has a nontrivial fundamental group.
    \item $\operatorname{Con}^\omega(G,(n_i))$ is locally simply connected for all $\omega$.
    \item $\operatorname{Con}^\omega(G,(d_i))$ has trivial fundamental group for all $\omega$.
\end{enumerate}

\end{cor}
\begin{proof}

Ol'shanskii and Sapir in \cite{OS2} showed that the second condition implies the existence of a $b<1$ such that $\frac{f(n_i)}{f(bn_i)}\to \infty.$  This was used to show that divisibility function restricted to $\bigcup\limits_i [bn_i,n_i]$ is unbounded.  Then (A) follows.

The first condition implies that $G$ is uniformly $\epsilon$-coarsely loop divisible for every $\epsilon>0$ with respect to the pair $\bigl(\omega, (d_i)\bigr)$ for any ultrafilter $\omega$.  Therefore $\operatorname{Con}^\omega(G,(d_i))$ has trivial fundamental group.

The second condition implies $\bigl($by the same argument that was used to show $G$ is uniformly $\epsilon$-coarsely loop divisible for every $\epsilon>0$ with respect to the pair $\bigl(\omega, (d_i)\bigr)$ that there exists an $\epsilon>0$ such that $G$ is uniformly $\epsilon$-coarsely loop divisible for the pair $\bigl(\omega, (n_i)\bigr)$.  Hence, $\operatorname{Con}^\omega(G,(n_i))$ is locally simply connected.

\end{proof}

\begin{quest}
    Can this group have an asymptotic cone which is not locally simply connected?
\end{quest}

Or for finitely presented groups in general:
\begin{quest}\label{question2}\ulabel{question2}{Question}
Is there a finitely presented group which has a locally simply connected asymptotic cone and an asymptotic cone which is not locally simply connected?  \end{quest}

\ref{question2} is false if we consider the larger class of finitely generated groups.  Thomas and Velicovick consider a group $G_I = \pres{a,b}{(a^nb^n)^7=1; \ n\in I}$ which they show for an appropriate choice of $I$ has a simply connected and non-simply connected asymptotic cone (see \cite{TV}).

\begin{prop}
Let $I_0 = \{2^{2^n}\}$. Let $I_{0,k} = I_0\cap [2^{2^k}, \infty)$ and $I= \bigcup_{k=1}^\infty 2^k\cdot I_{0,k}$.  Then $G_I = \pres{a,b}{(a^nb^n)^7 \ n\in I}$ has a cone which is locally simply connected and a cone which is not semi-locally simply connected.

\end{prop}

\begin{proof}

Let $\gamma_n$ be the loop based at the identity with label $(a^nb^n)^7$ for $n\in I$.  Then Thomas and Velicovick show using small cancelation that $P\bigl(\gamma_n, \frac{|\gamma_n|}{2} \bigr) = \infty$ (See Lemma 1.1 of \cite{TV}).

If we let $d_n = 8^{2^{n-1}}$, then the argument of Thomas and Velicovick shows that $\gcon$ is an $\R$-tree for any $\omega$.

Let $\rho_n = 2^n2^{2^n}= 2^{2^n +n}$.  Let $\gamma_{n,k}$ be the loop with label $(a^{2^{2^n +k}}b^{2^{2^n +k}})^7$ for $k\geq n$.  Then $\frac{|\gamma_{n,k}|}{\rho_n} = \frac {14}{2^{n-k}}$.  Hence $(\gamma_{n, n-i})$ is a loop of length $\frac {14}{2^{i}}$ in $\operatorname{Con}^\omega \bigr(G, (\rho_n)\bigl)$ which has no finite partition.  Hence $\operatorname{Con}^\omega \bigr(G, (\rho_n)\bigl)$ is not semi-locally simply connected and has uncountable fundamental group for any $\omega$.

\end{proof}

\textbf{Acknolwedgements.}  The authors would like to thank Mark Sapir for may helpful discussions and suggestions with the writing of this paper.  As well, we would like to thank Yves Cornulier for pointing out that the cones of solvable Baumslag-Solitar groups, $Sol_3$, and some extensions of $\mathbb R^n$ by $\mathbb R$ are bi-Lipschitz.

\newpage
\bibliographystyle{plain}
\bibliography{bib}

\end{document}